\documentclass[11pt]{article}

\usepackage[margin=1in]{geometry}
\usepackage{amsmath,amssymb,amsthm}
\usepackage{autonum}
\usepackage{hyperref}
\hypersetup{hidelinks}
\makeatletter
\AtBeginDocument{\autonum@generatePatchedReference{eqref}}
\makeatother

\newtheorem{theorem}{Theorem}
\newtheorem{lemma}[theorem]{Lemma}

\newtheorem{conjecture}[theorem]{Conjecture}
\newtheorem{corollary}[theorem]{Corollary}
\newtheorem{remark}[theorem]{Remark}
\newtheorem{question}[theorem]{Question}

\DeclareMathOperator{\Vol}{vol}
\newcommand{\R}{\mathbb{R}}
\newcommand{\Z}{\mathbb{Z}}
\newcommand{\covol}{\operatorname{covol}}
\newcommand{\dL}{\delta_{\mathrm L}}
\newcommand{\dT}{\delta_{\mathrm T}}
\newcommand{\dall}{\delta}
\newcommand{\alphaE}{\alpha}
\newcommand{\LamNine}{\Lambda_9}

\newcommand{\dd}{\,\mathrm{d}}

\title{On two counterexamples in the geometry of numbers}
\author{Nihar P. Gargava\\
}
\date{October 28, 2026}

\begin{document}
\maketitle

\begin{abstract}

We give counterexamples to two optimization
problems in dimensions eight and nine.

\begin{enumerate}
\item The Cartesian-product problem posed by Cassels for critical
determinants and later formulated by Zong for lattice packings and for
packings allowing translations but not rotations: whether the corresponding
product inequalities are always equalities \cite{Cassels1971,Zong2005}.
The counterexample also disproves the natural product formula
for unrestricted congruent packings.  
\item A question raised by Sarnak and formulated as a conjecture in Chiu
\cite{Chiu1997}: whether, among unit-volume flat tori, height is minimized by
a lattice maximizing the length of its shortest nonzero vector.
\end{enumerate}

\end{abstract}

\section{Introduction}

This paper gives counterexamples to two optimization principles for lattices.

\subsection{Cassels' problem of critical determinants}

The first concerns Cartesian products.  For an origin-symmetric convex body
$\mathfrak S$, its critical determinant is
\[
\Delta(\mathfrak S)
=
\inf\left\{
\covol(\Lambda):
(\Lambda\setminus\{0\})\cap\operatorname{int}(\mathfrak S)=\varnothing
\right\},
\]
where $\covol$ denotes covolume and the infimum is over full-rank lattices.  If
$\mathfrak T=\mathfrak K\times\mathfrak L$, products of admissible lattices give
\begin{equation}
        \Delta(\mathfrak T)
        \le
        \Delta(\mathfrak K)\Delta(\mathfrak L).
\end{equation}
Cassels displayed this inequality and asked whether it can ever be strict
\cite[pp.~227--228]{Cassels1971}.

Zong later asked whether, for arbitrary convex bodies, the density of a
product always equals the product of the two densities, both for lattice
packings and for packings allowing translations but not rotations
\cite[Problem~1, p.~74]{Zong2005}.  He stated the universal identities as
questions while explicitly expressing belief in them for
two-dimensional factors \cite[Remark~2, p.~77]{Zong2005}. 
This four-dimensional question is still open up to the author's knowledge.

We answer Cassels' question and Zong's all-dimensional questions using
$B=B_2^4$, the Euclidean unit ball in $\R^4$.  Writing $\dL$ for maximal
lattice packing density, the product value is
\[
        \dL(B)^2=\frac{\pi^4}{256},
\]
by the optimality of $D_4$ in dimension four \cite{KZ1872}, whereas a suitable rotation of the even unimodular lattice $E_8\subset\R^8$
\cite[Chapter~4]{CS1999}
gives a lattice packing of $B\times B$
with density
\[
        \frac{\pi^4}{1024}
        \left(1+\frac1{\sqrt5}\right)^4
        =0.417282401491719\ldots .
\]
This is about $9.67\%$ greater than the product value.  Equivalently,
\[
        \Delta((2B)\times(2B))<\Delta(2B)^2.
\]
The rotated $E_8$ packing, together with a rigorous upper bound 
for the
unrestricted four-dimensional sphere-packing density
\cite{CDLS2022,CDLSData},
also disproves the
corresponding translative and unrestricted product formulas.

\subsection{A question about heights of flat tori}

The second is a question about heights of flat tori. It was raised by Sarnak
and formulated as a conjecture in Chiu \cite{Chiu1997}.
It concerns the height of a lattice $\Lambda \subseteq \mathbb{R}^{n}$ defined 
as 
\begin{equation}
    h(\Lambda)=\zeta'_{\Delta_{\mathbb{R}^n/\Lambda}}(0),
\end{equation}
where $\Delta_{\mathbb{R}^n/\Lambda}$ is the nonnegative Laplacian and
\[
\zeta_{\Delta_{\mathbb{R}^n/\Lambda}}(s)
=
\sum_{\lambda\in\operatorname{Spec}(\Delta_{\mathbb{R}^n/\Lambda})\setminus\{0\}}
\lambda^{-s}
\]
is its spectral zeta function, where $\operatorname{Spec}$ denotes the spectrum
and $\lambda$ ranges over the nonzero eigenvalues of the Laplacian.  It is
continued meromorphically to $s=0$.

The conjecture states that
among
unit-volume flat tori, height should be minimized by a lattice maximizing the
length of its shortest nonzero vector.  The recent solution of the
nine-dimensional lattice-packing problem identifies the Korkine-Zolotareff lattice
$\LamNine$, up to isometry, as the unique such lattice
\cite{DutourVanWoerden2025}.  
We give an explicit one-parameter perturbation
$L_{t}$ of the lattice $L_{0}=\LamNine$ that preserves the determinant 
and
\[
   \left.\frac{\dd}{\dd t}h(L_t)\right|_{t=0}<0.
   \]
Thus this proves that $\LamNine$ is not stationary for
the height function $h(\cdot)$ and cannot be a global minimizer.  

This leads to two interesting questions which we leave for the future.
\begin{question}
    \begin{enumerate}
        \item 
Is dimension 9 the smallest dimension where this conjecture is false?
\item  Which unit covolume lattice minimizes the height function $h(\cdot)$ in dimension 9?
    \end{enumerate}
\end{question}

\begin{remark}
There have been recent results on Gaussian energy minimization in dimensions
three and four \cite{LuoWei3D,LuoWei4D}.  

A consequence of
Theorem~\ref{thm:height-descent} is that the packing-optimal lattice
$\LamNine$ does not minimize Gaussian energy for at least one positive
parameter. Hence it is impossible to keep extending \cite{LuoWei3D,LuoWei4D} in all dimensions.
\end{remark}

% \begin{remark}
%   We rely on 
%   \cite{CDLS2022} for the unrestricted version of Cassels' problem
%   and 
%   \cite{DutourVanWoerden2025} for our counterexample to the conjecture
%   formulated in Chiu.
%   At the time of writing, both \cite{CDLS2022} and \cite{DutourVanWoerden2025} are available only as preprints.
% \end{remark}

\section{The product problems of Cassels and Zong}

For convex bodies $K_1\subset\R^{n_1}$ and $K_2\subset\R^{n_2}$, their Cartesian
product
\[
        K_1\times K_2\subset\R^{n_1+n_2}
\]
is again a convex body.  Products of lattice or periodic packings of the two
factors give packings of $K_1\times K_2$.
Thus
\begin{equation}
        \delta_\star(K_1\times K_2)
        \ge
        \delta_\star(K_1)\,\delta_\star(K_2),
\end{equation}
where $\delta_\star$ denotes the optimal packing density in the class $\star$,
for example lattice packings, translative packings, or unrestricted congruent
packings.  The three versions considered here are stated below.

Cassels posed the lattice-restricted question in the older language of lattice
constants, or critical determinants.  In Cassels' notation, $\Delta(\mathfrak S)$ is the
infimum of $\covol(\Lambda)$ over lattices $\Lambda$ admissible for a set
$\mathfrak S$, meaning that no nonzero point of $\Lambda$ lies in
$\operatorname{int}(\mathfrak S)$.  If $\Lambda$ attains this infimum, it is called
critical.  Cassels introduces these notions in the prologue and again in Chapter
III, Section 5 \cite[pp.~6, 80]{Cassels1971}.

\begin{question}[Cassels]
Can the product inequality be strict?  That is, can
\[
        \Delta(\mathfrak K\times\mathfrak L)
        <\Delta(\mathfrak K)\Delta(\mathfrak L)
\]
hold for some origin-symmetric convex bodies $\mathfrak K$ and $\mathfrak L$?
\end{question}

This is the same lattice-packing question expressed in terms of critical
determinants.
Indeed, if $K\subset\R^n$ is centrally symmetric, then a lattice packing of
translates of $K$ with center lattice $\Lambda$ is equivalent to the condition
\[
        (\Lambda\setminus\{0\})\cap \operatorname{int}(2K)=\varnothing.
\]
Thus, with Cassels' normalization,
\begin{equation}\label{eq:delta-Delta-relation}
        \dL(K)=\frac{\Vol(K)}{\Delta(2K)}.
\end{equation}
Consequently
\[
        \dL(K_1\times K_2)=\dL(K_1)\dL(K_2)
\]
is equivalent to
\[
        \Delta(2K_1\times 2K_2)=\Delta(2K_1)\Delta(2K_2).
\]
A positive answer to Cassels' question is therefore exactly the same as a
strict improvement over the lattice product-density formula.

The corresponding packing questions were later posed explicitly in density
language by Zong.  Writing $\dT$ for maximal translative packing density and
$\dL$ for maximal lattice packing density, his Problem~1 is:
\begin{question}[Zong]
Do
\[
        \dT(K_1\times K_2)=\dT(K_1)\dT(K_2)
\]
and
\[
        \dL(K_1\times K_2)=\dL(K_1)\dL(K_2)
\]
hold for all convex bodies $K_1,K_2$?
\end{question}
See \cite[Problem~1, p.~74]{Zong2005}.
In Remark~2, however, he writes that
``we may believe'' both identities when $K_1$ and $K_2$ are two-dimensional
convex domains \cite[Remark~2, p.~77]{Zong2005}.  

We also consider the natural unrestricted analogue, where $\dall$ denotes the
supremum over all congruent packings, allowing arbitrary Euclidean motions.
This gives a third variant of Zong's Problem~1:
\begin{question}[Unrestricted analogue]
Does
\[
        \dall(K_1\times K_2)
        =
        \dall(K_1)\dall(K_2)
\]
hold for all convex bodies $K_1,K_2$?
\end{question}

The example below disproves all three packing formulae with
\[
        K_1=K_2=B_2^4.
\]
It is worth emphasizing that the unrestricted conclusion does not require knowing
the exact value of $\dall(B_2^4)$.  The construction gives a lower bound for
$\dall(B_2^4\times B_2^4)$, while the known rigorous upper bound for
$\dall(B_2^4)$ is already small enough to force a strict violation.

\subsection{Density normalization}

Let $K\subset\R^n$ be a centrally symmetric convex body, and let $\|\cdot\|_K$
be its Minkowski norm.  For a lattice $\Lambda\subset\R^n$, put
\[
        m_K(\Lambda)=\min_{0\ne x\in\Lambda}\|x\|_K.
\]
Here $\covol(\Lambda)$ denotes the covolume of $\Lambda$.  The largest
scaled copies of $K$ that can be centered at $\Lambda$ have scale
$m_K(\Lambda)/2$.  Equivalently, the rescaled center lattice
$(2/m_K(\Lambda))\Lambda$ packs unit copies of $K$.  The resulting density,
which is unchanged by uniformly rescaling $\Lambda$, is
\begin{equation}\label{eq:lattice-density-formula}
        \delta_K(\Lambda)
        =
        \frac{\Vol(K)}{\covol(\Lambda)}
        \left(\frac{m_K(\Lambda)}{2}\right)^n .
\end{equation}
Here $\Vol$ denotes Euclidean volume.

For
\[
        K=B_2^4\times B_2^4\subset\R^4\oplus\R^4,
\]
the relevant norm is the block norm
\begin{equation}
        \|(x,y)\|_K=\max\{\|x\|_2,\|y\|_2\},
\end{equation}
and
\begin{equation}\label{eq:volume-K}
        \Vol(K)=\Vol(B_2^4)^2
        =\left(\frac{\pi^2}{2}\right)^2
        =\frac{\pi^4}{4}.
\end{equation}

\subsection{The product \texorpdfstring{$D_4\oplus D_4$}{D4 plus D4}}

Let
\[
        D_4=\left\{x\in\Z^4:\sum_{i=1}^4 x_i\equiv0\pmod2\right\}.
\]
Then $\covol(D_4)=2$ and the shortest nonzero Euclidean length in $D_4$ is
$\sqrt2$.  Hence, for the product lattice $D_4\oplus D_4$ in the coordinate
decomposition $\R^8=\R^4\oplus\R^4$,
\[
        m_K(D_4\oplus D_4)=\sqrt2,
        \qquad
        \covol(D_4\oplus D_4)=4.
\]
Thus the actual center lattice for unit copies of $K$ is
$\sqrt2(D_4\oplus D_4)$.
Using \eqref{eq:lattice-density-formula} and \eqref{eq:volume-K},
\begin{equation}\label{eq:D4-product-density}
        \delta_K(D_4\oplus D_4)
        =
        \frac{\pi^4}{4}\cdot\frac1{4}
        \left(\frac{\sqrt2}{2}\right)^8
        =
        \frac{\pi^4}{256}.
\end{equation}

Korkine and Zolotareff proved that $D_4$ is the optimal lattice sphere packing in
dimension $4$ \cite{KZ1872}.  In our normalization,
\begin{equation}\label{eq:D4-factor-density}
        \dL(B_2^4)=\delta_{B_2^4}(D_4)
        =
        \frac{\pi^2}{16},
\end{equation}
so \eqref{eq:D4-product-density} is exactly the product value
\[
        \dL(B_2^4)^2=\frac{\pi^4}{256}.
\]

\subsection{A rotation of \texorpdfstring{$E_8$}{E8}}

Use the standard coordinate realization \cite[Chapter~4]{CS1999}
\[
E_8=
\left\{x\in\Z^8:\sum_i x_i\equiv0\pmod2\right\}
\cup
\left\{x\in\left(\Z+\frac12\right)^8:\sum_i x_i\equiv0\pmod2\right\}.
\]
This lattice is unimodular, even, and has minimal squared Euclidean norm $2$.
Its $240$ minimal vectors, or roots, are
\[
        (\pm1,\pm1,0^6)
\]
with arbitrary positions and signs, together with
\[
        (\pm\tfrac12,\ldots,\pm\tfrac12)
\]
with an even number of minus signs.

Consider the symmetric integer matrix
\[
H=
\begin{pmatrix}
-1&0&0&1&0&1&-1&-1\\
0&-1&0&-1&0&-1&-1&-1\\
0&0&-1&1&0&-1&-1&1\\
1&-1&1&1&1&0&0&0\\
0&0&0&1&-1&-1&1&-1\\
1&-1&-1&0&-1&1&0&0\\
-1&-1&-1&0&1&0&1&0\\
-1&-1&1&0&-1&0&0&1
\end{pmatrix}.
\]
Writing $I$ for the identity matrix and $\operatorname{tr}$ for the trace, a
direct calculation gives
\begin{equation}
        H^2=5I,
        \qquad
        \operatorname{tr}H=0.
\end{equation}
Therefore the $+\sqrt5$ and $-\sqrt5$ eigenspaces of $H$ are two orthogonal
$4$-planes.  Let $U$ be the $+\sqrt5$ eigenspace and let $V=U^\perp$.  The
orthogonal projection $P$ onto $U$ is
\begin{equation}\label{eq:projection}
        P=\frac12\left(I+\frac1{\sqrt5}H\right),
        \qquad
        I-P=\frac12\left(I-\frac1{\sqrt5}H\right).
\end{equation}
Let $\pi_1$ and $\pi_2$ be the orthogonal projections onto the two blocks in
the standard decomposition $\R^8=\R^4\oplus\R^4$.  Choose a rotation
$Q\in\mathrm{SO}(8)$ such that
\[
        Q(U)=\R^4\oplus\{0\},
        \qquad
        Q(V)=\{0\}\oplus\R^4.
\]
Then $QE_8$ is a rotated copy of $E_8$, and
\begin{equation}\label{eq:rotation-projections}
        \pi_1Q=QP,
        \qquad
        \pi_2Q=Q(I-P).
\end{equation}

\begin{lemma}[Root projections]\label{lem:root-projections}
For every root $r$ of $E_8$,
\[
        r^T H r=\pm2.
\]
Consequently
\[
        \|\pi_1Qr\|_2^2=1\pm\frac1{\sqrt5},
        \qquad
        \|\pi_2Qr\|_2^2=1\mp\frac1{\sqrt5}.
\]
\end{lemma}

\begin{proof}
For the $112$ roots $\pm e_i\pm e_j$, substitution gives $r^THr=\pm2$.
For the $128$ half-integer roots, write
$r=\frac12(\varepsilon_1,\ldots,\varepsilon_8)$, where
$\varepsilon_i\in\{\pm1\}$ and $\prod_i\varepsilon_i=1$; substitution
again gives $r^THr=\pm2$.  Given this assertion, the norm
identities are immediate from \eqref{eq:projection} and $\|r\|_2^2=2$:
\[
        \|Pr\|_2^2
        =r^TPr
        =\frac12\left(r^Tr+\frac1{\sqrt5}r^T H r\right)
        =1+\frac{r^T H r}{2\sqrt5}.
\]
The formula for $I-P$ is the same calculation with the opposite sign, and
\eqref{eq:rotation-projections} transfers these norms to the coordinate
projections of $Qr$.
\end{proof}

Set
\begin{equation}\label{eq:alpha-def}
        \alphaE=1+\frac1{\sqrt5}.
\end{equation}

\begin{lemma}[the $K$-minimum of the rotated $E_8$]\label{lem:E8-K-minimum}
For $K=B_2^4\times B_2^4$ in the standard coordinate decomposition,
\[
        m_K(QE_8)^2=\alphaE=1+\frac1{\sqrt5}.
\]
\end{lemma}

\begin{proof}
For every root $r$ of $E_8$, Lemma~\ref{lem:root-projections} gives
\[
        \max\{\|\pi_1Qr\|_2^2,\|\pi_2Qr\|_2^2\}
        =1+\frac1{\sqrt5}=\alphaE.
\]
Thus $m_K(QE_8)^2\le\alphaE$.

Conversely, if $0\ne x\in E_8$ is not a root, then $\|x\|_2^2\ge4$, because
$E_8$ is even and has minimal norm $2$.  Since $P$ and $I-P$ are complementary orthogonal
projections,
\[
        \|Px\|_2^2+\|(I-P)x\|_2^2=\|x\|_2^2.
\]
Hence one of the two squared block lengths is at least $2$, and
\[
        \max\{\|\pi_1Qx\|_2^2,\|\pi_2Qx\|_2^2\}
        =\max\{\|Px\|_2^2,\|(I-P)x\|_2^2\}
        \ge2>1+\frac1{\sqrt5}.
\]
Therefore the minimum is attained on roots and has squared value $\alphaE$.
\end{proof}

\begin{theorem}[density of the rotated $E_8$ packing]\label{thm:rotated-E8-density}
The center lattice $(2/\sqrt{\alphaE})QE_8$ gives a lattice packing of
$B_2^4\times B_2^4$ of density
\[
        \delta_K(QE_8)
        =
        \frac{\pi^4}{1024}\left(1+\frac1{\sqrt5}\right)^4
        =0.417282401491719\ldots .
\]
\end{theorem}

\begin{proof}
Rotation preserves covolume, so $\covol(QE_8)=\covol(E_8)=1$.  By
\eqref{eq:lattice-density-formula}, \eqref{eq:volume-K}, and
Lemma~\ref{lem:E8-K-minimum},
\[
        \delta_K(QE_8)
        =
        \frac{\pi^4}{4}
        \left(\frac{\sqrt{\alphaE}}{2}\right)^8
        =
        \frac{\pi^4}{1024}\alphaE^4.
\]
Substituting \eqref{eq:alpha-def} gives the stated expression.
\end{proof}

\subsection{The restricted counterexample}

\begin{theorem}[failure of the lattice-restricted product formula]
For $K_1=K_2=B_2^4$,
\[
        \dL(K_1\times K_2)>
        \dL(K_1)\dL(K_2).
\]
More precisely,
\[
        \dL(B_2^4\times B_2^4)
        \ge
        \frac{\pi^4}{1024}\left(1+\frac1{\sqrt5}\right)^4
        >
        \frac{\pi^4}{256}
        =\dL(B_2^4)^2.
\]
\end{theorem}

\begin{proof}
The lower bound is Theorem~\ref{thm:rotated-E8-density}.  The product value is
\eqref{eq:D4-factor-density}.  It remains only to compare the two explicit
numbers.  The ratio is
\[
        \frac{\frac{\pi^4}{1024}(1+1/\sqrt5)^4}{\frac{\pi^4}{256}}
        =
        \frac{(1+1/\sqrt5)^4}{4}.
\]
This is strictly greater than one: indeed
$(1+1/\sqrt5)^2>2$, since $2/\sqrt5>4/5$.  Numerically, the
ratio is $1.096656314599949\ldots$.
Thus the rotated $E_8$ lattice packing is about $9.67\%$ denser than the product
lattice value $D_4\oplus D_4$.
\end{proof}

\begin{corollary}[Cassels' critical-determinant form]
Let
\[
        \mathfrak S=2B_2^4.
\]
Then Cassels' product inequality is strict:
\[
        \Delta(\mathfrak S\times\mathfrak S)
        <
        \Delta(\mathfrak S)^2.
\]
More explicitly,
\[
        \Delta(\mathfrak S\times\mathfrak S)
        \le
        \frac{256}{(1+1/\sqrt5)^4}
        =58.359213\ldots
        <64
        =\Delta(\mathfrak S)^2.
\]
\end{corollary}

\begin{proof}
For a centrally symmetric body $K$, the relation \eqref{eq:delta-Delta-relation}
gives $\dL(K)=\Vol(K)/\Delta(2K)$.  Taking $K=B_2^4$ and using
$\dL(B_2^4)=\pi^2/16$ gives
\[
        \Delta(2B_2^4)=\frac{\Vol(B_2^4)}{\dL(B_2^4)}
        =\frac{\pi^2/2}{\pi^2/16}=8.
\]
Thus $\Delta(\mathfrak S)^2=64$.  On the other hand, the rotated $E_8$ packing
of $B_2^4\times B_2^4$ from Theorem~\ref{thm:rotated-E8-density} gives
\[
        \Delta(2(B_2^4\times B_2^4))
        \le
        \frac{\Vol(B_2^4\times B_2^4)}{\delta_K(QE_8)}
        =
        \frac{\pi^4/4}{(\pi^4/1024)(1+1/\sqrt5)^4}
        =
        \frac{256}{(1+1/\sqrt5)^4}.
\]
Since $2(B_2^4\times B_2^4)=(2B_2^4)\times(2B_2^4)=\mathfrak S\times\mathfrak S$,
the displayed inequality proves the claim.
\end{proof}

\subsection{The unrestricted counterexample}

Now let $\dall$ denote unrestricted packing density, allowing arbitrary congruent
copies.  The rotated $E_8$ construction is a lattice packing by translates of one
fixed orientation of $B_2^4\times B_2^4$, and is therefore also an admissible
unrestricted packing.  Hence
\begin{equation}\label{eq:unrestricted-product-lower}
        \dall(B_2^4\times B_2^4)
        \ge
        \frac{\pi^4}{1024}\left(1+\frac1{\sqrt5}\right)^4
        =0.417282401491719\ldots .
\end{equation}

For the factor $B_2^4$, the exact unrestricted sphere-packing density is not
needed here.  Cohn's sphere-packing data table lists the density
\[
        0.6168502750680849\ldots
\]
coming from $D_4$ \cite{CohnTable}.  Cohn--de Laat--Salmon rigorously certified by Arb interval
arithmetic that the unrestricted density is at most $0.636107333$
\cite{CDLS2022,CDLSData}.  Their
archived verifier gives the upper endpoint
$0.63610733215513285089\ldots$, which Cohn's table records with outward
rounding as
\begin{equation}\label{eq:four-dimensional-upper-bound}
        \dall(B_2^4)
        \le
        0.6361073321551329.
\end{equation}
See \cite{CDLS2022,CDLSData,CohnTable}.

For a comparison that does not depend on rounded display digits, use
$\pi>3.14$ and $\sqrt5<25/11$, the latter of which implies
$1+1/\sqrt5>1.44$.  Then \eqref{eq:unrestricted-product-lower} is greater
than
\[
        \frac{3.14^4\,1.44^4}{1024}
        =0.4081958678\ldots>0.408.
\]
On the other hand, \eqref{eq:four-dimensional-upper-bound} is less than
$0.637$, and $0.637^2=0.405769<0.408$.  Consequently
\[
        \dall(B_2^4\times B_2^4)>\dall(B_2^4)^2.
\]
Thus:

\begin{theorem}[failure of the unrestricted product formula]
For $K_1=K_2=B_2^4$,
\[
        \dall(K_1\times K_2)>
        \dall(K_1)\dall(K_2).
\]
\end{theorem}

\begin{remark}[translative density]
The same argument also disproves the intermediate translative product formula
\[
        \dT(K_1\times K_2)\stackrel{?}{=}\dT(K_1)\dT(K_2)
\]
for $K_1=K_2=B_2^4$.  Indeed, the rotated $E_8$ packing is translative, and for a
Euclidean ball translative and unrestricted packings have the same factor density
because rotations do not change the body.
\end{remark}

\section{A question about heights of flat tori}

\subsection{Flat tori, packing, and height}

Let $L=C\Z^n\subset\R^n$ be a lattice with Gram matrix
$G=C^{\mathsf T}C$.  The associated flat torus is
\[
        T_L=\R^n/L,
        \qquad
        \Vol(T_L)=\sqrt{\det G}.
\]
We normalize throughout this section by $\det G=1$.  The squared first minimum
is
\[
        \lambda_1(L)^2
        =\min_{m\in\Z^n\setminus\{0\}}m^{\mathsf T}Gm.
\]
At fixed covolume, maximizing $\lambda_1(L)$ is equivalent to maximizing the
density of the associated lattice sphere packing.

Let $\Delta_G$ denote the nonnegative Laplacian on $T_L$.  Its spectral zeta
function, initially defined for $\Re s>n/2$, is
\[
        \zeta_{\Delta_G}(s)
        =\sum_{\lambda\in\operatorname{Spec}(\Delta_G)\setminus\{0\}}
        \lambda^{-s}.
\]
It has a meromorphic continuation that is regular at $s=0$.  The
zeta-regularized determinant and the height are
\[
        \det{}'_{\zeta}\Delta_G
        =\exp\bigl(-\zeta'_{\Delta_G}(0)\bigr),
        \qquad
        h(G)=\zeta'_{\Delta_G}(0).
\]
Thus minimizing height is equivalent to maximizing the regularized determinant.

Chiu formulates the question as follows
\cite[Conjecture~4.5]{Chiu1997}.

\begin{conjecture}[Question about heights of flat tori]
Among all $n$-dimensional flat tori of volume one, the height is minimized by a
torus corresponding to a lattice maximizing the length of its shortest
nonzero vector.
\end{conjecture}

The conjecture is known in dimension two, where the hexagonal torus is the
global minimizer, and in dimension three, where the face-centered cubic torus
is the unique global minimizer \cite{OPS1988,SarnakStrombergsson2006}.
In dimensions eight and twenty-four, universal optimality of the $E_8$ and
Leech lattices implies both their sphere-packing optimality and the
height-minimizing property, and hence gives an affirmative answer in those
dimensions \cite[Corollary 1.5]{CKMRV2019}.
Francesconi and Rowlett proved the conjecture within the class of
orthogonal flat tori in every dimension \cite{FrancesconiRowlett2026}.
\subsection{The densest lattice in dimension nine}

Dutour Sikiri\'c and van Woerden recently completed Voronoi's algorithm in
dimension nine, enumerating $2{,}237{,}251{,}040$ perfect forms and
$7{,}338{,}582$ extreme forms \cite{DutourVanWoerden2025}.
The lattice $\LamNine$ is called the Korkine-Zolotareff lattice
\cite[p.~xvii]{CS1999}.

\begin{theorem}[Dutour Sikiri\'c--van Woerden]
Up to similarity, the Korkine-Zolotareff lattice $\LamNine$ is the unique densest
nine-dimensional lattice packing.  With determinant normalized to one, its
shortest nonzero vectors have squared length $2$.
\cite{DutourVanWoerden2025}
\end{theorem}

Let
\[
D_8=\left\{z\in\mathbb Z^8:\sum_{i=1}^8 z_i\equiv0\pmod2\right\},
\qquad
E_8=D_8\cup\left(D_8+\frac12(1,\ldots,1)\right),
\]
and let $e_1=(1,0,\ldots,0)\in\mathbb R^8$.  We use the following
description of $\LamNine$ \cite[Example~1]{DutourVanWoerden2025}:
\[
\LamNine=\{(x+ke_1,k):x\in E_8,\ k\in\mathbb Z\}.
\]
Since $E_8$ has covolume one, so does $\LamNine$.

This supplies the coordinates needed for the perturbation below.

\subsection{An explicit perturbation that lowers height}

For $t\in\mathbb R$, define
\[
L_t=\{(e^t(x+ke_1),e^{-8t}k):x\in E_8,\ k\in\mathbb Z\}.
\]
Then $L_0=\LamNine$, and each $L_t$ has covolume one.

\begin{theorem}\label{thm:height-descent}
The derivative of the height along this deformation satisfies
\[
\left.\frac{\dd}{\dd t}h(L_t)\right|_{t=0}<0.
\]
In particular, $\LamNine$ is not a local minimizer of height.
\end{theorem}

\begin{proof}
For $c>0$, define
\[
A(c)=\sum_{m\in\mathbb Z}e^{-cm^2},\qquad
B(c)=\sum_{m\in\mathbb Z}(-1)^m e^{-cm^2},\qquad
C(c)=\sum_{m\in\mathbb Z}e^{-c(m+1/2)^2}.
\]
Write
\[
\Theta_L(c)=\sum_{v\in L}e^{-c|v|^2}.
\]
Since $2e_1\in E_8$, the even and odd layers of $L_t$ give
\[
\Theta_{L_t}(c)
=\frac12\left[
 \bigl(A(x)^8+C(x)^8\bigr)A(z)+B(x)^8B(z)
\right],
\qquad
x=ce^{2t},\quad z=ce^{-16t}.
\]
Differentiating at $t=0$ gives
\[
\left.\frac{\dd}{\dd t}\Theta_{L_t}(c)\right|_{t=0}
=8cC(c)^7\bigl(A(c)C'(c)-C(c)A'(c)\bigr).
\]
The Jacobi identity (which follows from \cite[(20.4.6), (20.7.4), and
(20.13.1)]{DLMF})
\[
\frac{\dd}{\dd c}\log\frac{C(c)}{A(c)}
=-\frac14B(c)^4
\]
therefore yields
\[
\left.\frac{\dd}{\dd t}\Theta_{L_t}(c)\right|_{t=0}
=-2cA(c)B(c)^4C(c)^8<0.
\]

For covolume-one lattices $L$ and $K$, the spectral zeta function of
$\mathbb R^9/L$ is expressed using the dual lattice $L^*$.  Mellin
transformation and Poisson summation give
\[
h(L)-h(K)
=\pi^{-9/2}\int_0^\infty
c^{7/2}\bigl(\Theta_L(c)-\Theta_K(c)\bigr)\,\dd c.
\]
The integral converges at infinity because the nonzero lattice vectors have
positive length.  
At zero, Poisson summation cancels the common leading
term for the two covolume-one lattices.

Applying this formula to $L_t$ and $L_0$, and differentiating under the
integral, gives
\[
\left.\frac{\dd}{\dd t}h(L_t)\right|_{t=0}
=-2\pi^{-9/2}\int_0^\infty
c^{9/2}A(c)B(c)^4C(c)^8\,\dd c<0.
\]
\end{proof}

\begin{remark}
        An intuitive explanation for why~Theorem \ref{thm:height-descent} is true is the following:
height depends on vectors of all lengths in both lattice $\Lambda_{9}$ and the dual of $\Lambda_{9}$. On the other hand, packing
density depends only on the shortest vectors in $\Lambda_{9}$. 
The lattice $\Lambda_{9}$ has a dual whose shortest vector is shorter than that of $\Lambda_{9}$.

\end{remark}

\section{Acknowledgements}

\subsection{AI interaction}

The counterexamples were found using ChatGPT 5.5 Pro in various chat sessions.
The final paper is partly the author's own writing and partly written with the help of Codex.

For Cassels' problem, the author was expecting to find a counterexample using the lattice $E_8 \times E_8 \times E_8 \subseteq \mathbb{R}^{24}$ against the Leech lattice $\Lambda_{24} \subseteq \mathbb{R}^{24}$ with a product of three $8$-dimensional balls, but this was computationally too expensive to check. 
Then the author prompted the model to test $D_{4} \times D_{4} \subseteq \mathbb{R}^{8}$ against $E_{8} \subseteq \mathbb{R}^{8}$. 
The correct rotation for the $E_8$ lattice was from the model itself.

In investigating height minimization for flat tori, the author was expecting to find a counterexample in dimension 5, 6, 7 or 9. 
The model developed the numerical methods to compute heights of lattices by itself and the algorithm was tested for many different perfect forms in these dimensions to search for a counterexample.
Finally the counterexample was found by the model with the prompt ``Can you test small local deformations of $\Lambda_9$?''
The right perturbation to deform $\Lambda_{9}$ to obtain a counterexample was found by the model.
Initially, this finding was numerical but later ChatGPT 6.0 Pro found a rigorous proof using theta functions.

Chat logs of the sessions that found the counterexamples
can be given on request. 

\subsection{Human interactions}

The author would like to thank Anurag Rao for informing him about Cassels' problem and for some feedback on an earlier draft.

The author would also like to thank Danylo Radchenko for some pre-AI discussions about the heights of lattices. It was Danylo's suggestion that if the conjecture
is false, dimensions 5,6 or 7 might be worth looking at. Dimension 9 had only become available recently.
The author also thanks Peter Sarnak for some feedback on an earlier draft.

\subsection{Funding}

The author is funded by the Swiss National Science Foundation (SNSF) through
the project ``Random Geometry with Arithmetic Constraints'' (grant
no.~225437).

\bigskip
\noindent Nihar P. Gargava\\
Institut de Math\'ematiques d'Orsay, Universit\'e Paris-Saclay, France\\
\texttt{nihar.gargava@universite-paris-saclay.fr}

\end{document}